\documentclass[11pt]{amsart}
\usepackage[margin=1.1in]{geometry}
\usepackage{amsmath,amssymb,amsthm}
\usepackage{needspace}
\usepackage[colorlinks,linkcolor=blue,citecolor=blue,urlcolor=blue]{hyperref}

\numberwithin{equation}{section}
\newtheorem{maintheorem}{Theorem}

\newtheorem{proposition}{Proposition}[section]
\newtheorem{lemma}[proposition]{Lemma}
\theoremstyle{remark}

\theoremstyle{plain}

\newcommand{\dd}{\mathrm d}
\newcommand{\dstar}{\mathrm d^*}
\newcommand{\DeltaH}{\Delta_H}
\newcommand{\R}{\mathcal R}
\newcommand{\Ric}{\operatorname{Ric}}
\newcommand{\Id}{\operatorname{Id}}

\title{Rigidity in sharp eigenvalue estimates for differential forms}
\author{Tsz-Kiu Aaron Chow}
\address{Department of Mathematics, The Hong Kong University of Science and Technology, Clear Water Bay, Kowloon, Hong Kong}
\email{chowtka@ust.hk}

\author{James Koh}
\address{Department of Mathematics, The Hong Kong University of Science and Technology, Clear Water Bay, Kowloon, Hong Kong}
\email{jkoh@connect.ust.hk}

\author{Sen Yuk Wu}
\address{Department of Mathematics, The Hong Kong University of Science and Technology, Clear Water Bay, Kowloon, Hong Kong}
\email{sywuab@connect.ust.hk}

\thanks{T.-K. A. C. is partially supported by the National Natural Science Foundation of China under the Young Scientists Fund (Category C), Grant No.~NSFC26SC44.}

\expandafter\providecommand\csname subjclassname@2020\endcsname{\textup{2020} Mathematics Subject Classification}

\begin{document}

\begin{abstract}
For \mbox{$1\leq p\leq n/2$}, we derive sharp eigenvalue estimates for the Hodge Laplacian on differential $p$-forms under the curvature condition introduced by Petersen and Wink, and study their equality cases. Equality recovers the stronger curvature operator lower bound $\R\geq\Id$ and, in the closed case, leaves only round or Sasakian metrics on the universal cover. For manifolds with boundary, weak $p$-convexity of the boundary makes the unit hemisphere the only equality case for the absolute boundary problem.
\end{abstract}

\maketitle

\section{Introduction}

A basic question in Riemannian geometry is whether equality in a sharp eigenvalue estimate determines the underlying metric. On a closed connected $n$-dimensional manifold with \mbox{$\Ric\geq(n-1)g$}, Lichnerowicz's estimate bounds the first positive eigenvalue of the Laplacian on functions from below by $n$, and Obata's theorem \cite{Obata} characterizes equality by the unit round sphere. For differential forms, Gallot and Meyer \cite{GM} obtained the sharp lower bound $p(n-p+1)$ for the Hodge Laplacian in degrees $1\leq p\leq n/2$ when the curvature operator is bounded below by the identity. Under this curvature bound, the equality eigenspaces were described in terms of conformal and special Killing forms in \cite{TachibanaYamaguchi,Kora}. Equality in higher degrees need not imply roundness, as shown by Gallot's nonround Sasakian examples \cite[Counterexample~6.4]{Gallot}.

For two-forms, a sharp eigenvalue estimate under a positive lower bound for isotropic curvature was proved in \cite{ChowWang}, where the conformal Killing equations relate spectral equality to geometric rigidity. We consider the corresponding questions for differential $p$-forms under the curvature condition of Petersen and Wink.

Let $\R:\Lambda^2TM\to\Lambda^2TM$ be the curvature operator normalized so that the unit sphere has $\R=\Id$, and denote its eigenvalues by $\mu_1\leq\cdots\leq\mu_{\binom n2}$. Petersen and Wink \cite{PW} proved that, on a closed manifold, positivity of $\sum_{j=1}^{n-p}\mu_j$ implies vanishing of the $p$th Betti number for $1\leq p\leq n/2$. This condition is weaker than positivity of the curvature operator, but still controls the curvature term in the Weitzenb\"ock formula. We assume the normalized lower bound
\begin{equation}\label{eq:curvature}
\sum_{j=1}^{n-p}\mu_j\geq n-p
\end{equation}
at every point of $M$. Under this hypothesis, we derive the Gallot--Meyer eigenvalue bound on closed manifolds and study the geometric consequences of equality.

Let $\DeltaH=\dd\dstar+\dstar\dd$ denote the Hodge Laplacian. On a closed manifold, its first eigenvalue on $p$-forms is
\[
\lambda_1^{(p)}(M,g)=\inf_{0\ne\omega\in\Omega^p(M)}\frac{\displaystyle\int_M\bigl(|\dd\omega|^2+|\dstar\omega|^2\bigr)\,dV_g}{\displaystyle\int_M|\omega|^2\,dV_g}.
\]

\begin{maintheorem}\label{thm:closed}
Let $(M^n,g)$ be closed and connected, with $n\geq3$, and let $1\leq p\leq n/2$ be an integer. If \eqref{eq:curvature} holds, then
\[
\lambda_1^{(p)}(M,g)\geq p(n-p+1).
\]
Equality implies $\R\geq\Id$. If $p=1$, then $M$ is isometric to the unit sphere. If $p\geq2$, then its universal cover is diffeomorphic to $S^n$ and is either the unit round sphere or a nonround Sasakian manifold. In the nonround case, $n$ is odd and $p$ is even.
\end{maintheorem}

Gallot's examples show that the nonround alternative in Theorem~\ref{thm:closed} cannot be omitted. We next consider the absolute eigenvalue problem on manifolds with boundary. The Gallot--Meyer estimate under $\R\geq\Id$ extends to absolute boundary conditions when the boundary is weakly $p$-convex \cite[Remark~2.2]{GueriniSavo}. We obtain the same bound under the weaker interior curvature condition \eqref{eq:curvature}. Surprisingly, equality forces the unit hemisphere, excluding the nonround alternative that occurs in the closed case.

Let $M$ be compact with smooth boundary $\Sigma=\partial M$, let $\nu$ denote the outward unit normal, and let $A(X,Y)=\langle\nabla_X\nu,Y\rangle$ be the second fundamental form of $\Sigma$. Denote its principal curvatures by $\kappa_1\leq\cdots\leq\kappa_{n-1}$. The boundary is called weakly $p$-convex if
\begin{equation}\label{eq:pconvex}
\kappa_1+\cdots+\kappa_p\geq0
\end{equation}
at every point of $\Sigma$. The first eigenvalue of the Hodge Laplacian with absolute boundary conditions is given by
\[
\lambda_{1,\mathrm{abs}}^{(p)}(M,g)=\inf_{\substack{0\ne\omega\in\Omega^p(M)\\\iota_\nu\omega=0\text{ on }\Sigma}}\frac{\displaystyle\int_M\bigl(|\dd\omega|^2+|\dstar\omega|^2\bigr)\,dV_g}{\displaystyle\int_M|\omega|^2\,dV_g}.
\]
Here $\iota_\nu\omega=0$ means that $\omega$ has no normal component on $\Sigma$.

\begin{maintheorem}\label{thm:boundary}
Let $(M^n,g)$ be compact and connected with nonempty smooth boundary, with $n\geq3$, and let $1\leq p\leq n/2$ be an integer. Suppose that \eqref{eq:curvature} holds and the boundary is weakly $p$-convex. Then
\[
\lambda_{1,\mathrm{abs}}^{(p)}(M,g)\geq p(n-p+1).
\]
Equality holds if and only if $M$ is isometric to the unit hemisphere.
\end{maintheorem}

We derive the eigenvalue estimates by combining the pointwise curvature estimate of Petersen and Wink \cite{PW} with the integral formula for differential forms in \cite{ChowZhu} and the orthogonal decomposition of the gradient term~\eqref{eq:orthogonal-decomposition}. At equality, the twistor term in this decomposition vanishes, so the equality eigenforms are conformal Killing. The main point of our rigidity argument is that these overdetermined equations interact with \eqref{eq:curvature} to recover the stronger curvature operator bound $\R\geq\Id$.

For $2\leq p<n/2$, an equality eigenform $\omega$ is closed, and its canonical coclosed primitive is $\alpha=[p(n-p+1)]^{-1}\dstar\omega$, with $\dd\alpha=\omega$. A second application of the Bochner identity, now in degree $p-1$, forces equality in the corresponding curvature estimate and gives
\[
\nabla_X\alpha=\frac1p\iota_X\omega,\qquad
\nabla_X\omega=-pX^\flat\wedge\alpha.
\]
Thus $\alpha$ is a special Killing form. Equality in the adjacent degree, together with \eqref{eq:curvature}, forces $\R\geq\Id$. The weaker condition ~\eqref{eq:curvature} therefore recovers the full curvature operator bound in the equality case. In middle degree $2p=n$, a suitable closed or coclosed equality component yields the corresponding special Killing system and the same curvature conclusion; see Propositions~\ref{prop:closed-pair} and~\ref{prop:boundary-pair}.

This mechanism has a close analogue in spin geometry. The Weitzenb\"ock formula and the decomposition of the spinor gradient term into its Dirac and twistor components yield Friedrich's inequality for the Dirac operator under a positive scalar curvature lower bound. At equality, the eigenspinor is a real Killing spinor \cite{Friedrich}. On compact spin manifolds with boundary, an analogue of Friedrich's inequality holds under suitable local boundary conditions, with equality yielding real Killing spinors and a minimal boundary \cite{HMZBoundary}. Conformal Killing forms play the role of twistor spinors in the decomposition for differential forms, while special Killing forms, like real Killing spinors, correspond after normalization to parallel objects on the metric cone \cite{Semmelmann,Bar}. This analogy concerns the equality equations and their geometric consequences: a real Killing spinor forces the metric to be Einstein, although not necessarily round, whereas special Killing forms also occur on non-Einstein Sasakian manifolds. In our closed problem, the recovered bound $\R\geq\Id$ and the classification of special Killing forms leave only the round and nonround Sasakian alternatives on the universal cover described in Theorem~\ref{thm:closed}.

The boundary problem imposes a further restriction on this equality structure. For $p\geq2$, the resulting coupled equations and the absolute boundary conditions, together with weak $p$-convexity, force $A\geq0$. Moreover, the associated parallel form on the metric cone annihilates the normal to the cone boundary. Its kernel defines a parallel distribution containing that normal, which, together with the recovered curvature bounds, leads to hemisphere rigidity. Thus spectral equality strengthens both the interior and boundary curvature assumptions, and the boundary conditions exclude the nonround equality geometries of the closed problem.

For $p=1$, the equality system reduces to Obata's equation, with the boundary case governed by the classical Neumann rigidity theorem \cite{Escobar}. Special Killing forms on the boundary were also used in \cite{RS} to prove Euclidean-ball rigidity under nonnegative interior curvature and positive lower bounds on sums of principal curvatures. Here the special Killing equations arise from spectral equality in the interior, and weak $p$-convexity allows the totally geodesic boundary of the hemisphere.

\section{The eigenvalue estimate and its equality case}

\subsection{The eigenvalue estimates}

Let $(M^n,g)$ be compact and connected, with boundary $\Sigma$ possibly empty. We write $\|\cdot\|$ for the $L^2(M)$ norm and denote by $\mathcal Q$ the curvature term in the Weitzenb\"ock formula
\[
(\dd\dstar+\dstar\dd)\omega=\DeltaH\omega=\nabla^*\nabla\omega+\mathcal Q(\omega).
\]
For a $p$-form $\omega$, define
\[
P_X\omega=\nabla_X\omega-\frac1{p+1}\iota_X\dd\omega+\frac1{n-p+1}X^\flat\wedge\dstar\omega.
\]
The form $\omega$ is conformal Killing if $P\omega=0$. The orthogonal decomposition of $\nabla\omega$ \cite[Section~2]{Semmelmann} gives
\begin{equation}\label{eq:orthogonal-decomposition}
|\nabla\omega|^2=|P\omega|^2+\frac1{p+1}|\dd\omega|^2+\frac1{n-p+1}|\dstar\omega|^2.
\end{equation}

\begin{lemma}\label{lem:integral}
Let $1\leq p\leq n-1$, and let $\omega$ be a smooth $p$-form satisfying $\iota_\nu\omega=0$ on $\Sigma$. Then
\[
\begin{aligned}
&\frac{p}{p+1}\|\dd\omega\|^2+\frac{n-p}{n-p+1}\|\dstar\omega\|^2\\
&=\|P\omega\|^2+\int_M\langle\mathcal Q(\omega),\omega\rangle+\int_\Sigma\sum_{i_1<\cdots<i_p}(\kappa_{i_1}+\cdots+\kappa_{i_p})|\omega(e_{i_1},\ldots,e_{i_p})|^2,
\end{aligned}
\]
where $e_1,\ldots,e_{n-1}$ is an orthonormal frame at a boundary point with $A(e_a,e_b)=\kappa_a\delta_{ab}$. 
\end{lemma}

\begin{proof}
Let $e^1,\ldots,e^{n-1}$ be dual to $e_1,\ldots,e_{n-1}$. Since $\iota_\nu\omega=0$, \cite[Corollary~3.4 and Lemma~3.6]{ChowZhu} gives
\[
\begin{aligned}
&\|\dd\omega\|^2+\|\dstar\omega\|^2\\
&=\|\nabla\omega\|^2+\int_M\langle\mathcal Q(\omega),\omega\rangle+\int_\Sigma\sum_{i,j=1}^{n-1}A(e_i,e_j)\langle e^i\wedge\iota_{e_j}\omega,\omega\rangle.
\end{aligned}
\]
Since $A(e_a,e_b)=\kappa_a\delta_{ab}$, we have
\[
\begin{aligned}
\sum_{i,j=1}^{n-1}A(e_i,e_j)\langle e^i\wedge\iota_{e_j}\omega,\omega\rangle&=\sum_{a=1}^{n-1}\kappa_a\langle e^a\wedge\iota_{e_a}\omega,\omega\rangle\\
&=\sum_{a=1}^{n-1}\kappa_a|\iota_{e_a}\omega|^2\\
&=\sum_{a=1}^{n-1}\kappa_a\sum_{j_1<\cdots<j_{p-1}}|\omega(e_a,e_{j_1},\ldots,e_{j_{p-1}})|^2\\
&=\sum_{a=1}^{n-1}\kappa_a\sum_{\substack{i_1<\cdots<i_p\\a\in\{i_1,\ldots,i_p\}}}|\omega(e_{i_1},\ldots,e_{i_p})|^2\\
&=\sum_{i_1<\cdots<i_p}(\kappa_{i_1}+\cdots+\kappa_{i_p})|\omega(e_{i_1},\ldots,e_{i_p})|^2.
\end{aligned}
\]
Together with \eqref{eq:orthogonal-decomposition}, this gives
\[
\begin{aligned}
&\frac{p}{p+1}\|\dd\omega\|^2+\frac{n-p}{n-p+1}\|\dstar\omega\|^2\\
&=\|\dd\omega\|^2+\|\dstar\omega\|^2-\|\nabla\omega\|^2+\|P\omega\|^2\\
&=\|P\omega\|^2+\int_M\langle\mathcal Q(\omega),\omega\rangle+\int_\Sigma\sum_{i_1<\cdots<i_p}(\kappa_{i_1}+\cdots+\kappa_{i_p})|\omega(e_{i_1},\ldots,e_{i_p})|^2.
\end{aligned}
\]
\end{proof}

For $1\leq p\leq n/2$, the algebraic curvature operator $\R-\frac1{n-p}(\sum_{j=1}^{n-p}\mu_j)\Id$ has its lowest $n-p$ eigenvalues summing to zero. Applying \cite[Lemmas~2.1--2.2 and Proposition~2.5(a)]{PW} to this operator gives
\begin{equation}\label{eq:PW}
\langle\mathcal Q(\omega),\omega\rangle\geq p\left(\sum_{j=1}^{n-p}\mu_j\right)|\omega|^2.
\end{equation}
If \eqref{eq:curvature} holds, the ordering of the eigenvalues gives
\begin{equation}\label{eq:ordering}
\mu_{n-p}\geq1,\qquad \sum_{j=1}^k\mu_j\geq n-p+(k-n+p)\mu_{n-p}\geq k\quad(n-p\leq k\leq n-1).
\end{equation}
In particular, \eqref{eq:PW} in degree one gives $\Ric\geq(n-1)g$.

Suppose $1\leq p\leq n/2$, \eqref{eq:curvature} holds, and $\Sigma$ is weakly $p$-convex when nonempty. For every smooth $p$-form $\omega$ with $\iota_\nu\omega=0$, Lemma~\ref{lem:integral} and \eqref{eq:PW} give
\begin{equation}\label{eq:spectral-bound}
\begin{aligned}
&\frac{n-p}{n-p+1}\left(\|\dd\omega\|^2+\|\dstar\omega\|^2-p(n-p+1)\|\omega\|^2\right)\\
&\geq\|P\omega\|^2+\int_M\left(\langle\mathcal Q(\omega),\omega\rangle-p(n-p)|\omega|^2\right)+\frac{n-2p}{(p+1)(n-p+1)}\|\dd\omega\|^2\\
&\geq0.
\end{aligned}
\end{equation}
Taking the Rayleigh infimum proves the eigenvalue bounds in Theorems~\ref{thm:closed} and~\ref{thm:boundary}. At equality, every corresponding eigenform satisfies
\[
P\omega=0,\qquad \mathcal Q(\omega)=p(n-p)\omega,
\]
and $\dd\omega=0$ if $2p<n$.

\subsection{Equality in the eigenvalue estimate}

\begin{lemma}\label{lem:curvature-equality}
Let $(M^n,g)$ be a Riemannian manifold, possibly with boundary, and let $2\leq p\leq n/2$. Suppose that \eqref{eq:curvature} holds. If a smooth $(p-1)$-form $\alpha$ is nonzero on a dense subset of $M$ and satisfies
\[
\langle\mathcal Q(\alpha),\alpha\rangle=(p-1)(n-p+1)|\alpha|^2,
\]
then $\R\geq\Id$.
\end{lemma}

\begin{proof}
At every point where $\alpha\ne0$, \eqref{eq:PW} in degree $p-1$ and \eqref{eq:ordering} give
\[
\begin{aligned}
n-p+1=\frac{\langle\mathcal Q(\alpha),\alpha\rangle}{(p-1)|\alpha|^2}\geq\sum_{j=1}^{n-p+1}\mu_j\geq n-p+\mu_{n-p+1}\geq n-p+1.
\end{aligned}
\]
Thus $\mu_{n-p+1}=1$. Since $\mu_j\leq1$ for $j\leq n-p$ and $\sum_{j=1}^{n-p}\mu_j\geq n-p$, we have $\mu_1=\cdots=\mu_{n-p+1}=1$. Hence $\R\geq\Id$ wherever $\alpha\ne0$, and therefore on all of $M$ by continuity.
\end{proof}

\begin{proposition}\label{prop:closed-pair}
Under the hypotheses of Theorem~\ref{thm:closed}, suppose that $p\geq2$ and $\lambda_1^{(p)}(M,g)=p(n-p+1)$. Then $\R\geq\Id$, and the universal cover admits nonzero forms $\alpha\in\Omega^{p-1}(\widetilde M)$ and $\omega\in\Omega^p(\widetilde M)$ satisfying
\[
\nabla_X\alpha=\frac1p\iota_X\omega,\qquad \nabla_X\omega=-pX^\flat\wedge\alpha
\]
for every vector field $X$.
\end{proposition}

\begin{proof}
Since $\Ric\geq(n-1)g$, the universal cover $\widetilde M$ is compact by Bonnet--Myers. Let $\omega$ be the pullback of a nonzero eigen-$p$-form on $M$ with eigenvalue $p(n-p+1)$. If $2p<n$, then $\dd\omega=0$ by \eqref{eq:spectral-bound}. Suppose now that $n=2p$, and choose an orientation of $\widetilde M$. If $\dd\dstar\omega\ne0$, then
\[
\dd(\dd\dstar\omega)=0,\qquad \DeltaH(\dd\dstar\omega)=\dd\dstar\DeltaH\omega=p(p+1)\dd\dstar\omega,
\]
and we replace $\omega$ by $[p(p+1)]^{-1}\dd\dstar\omega$. Otherwise, Green's formula and $\DeltaH\star=\star\DeltaH$ give
\[
\begin{aligned}
\|\dstar\omega\|^2&=\int_{\widetilde M}\langle\dd\dstar\omega,\omega\rangle=0,\\
\dd(\star\omega)&=(-1)^p\star\dstar\omega=0,\\
\DeltaH(\star\omega)&=\star\DeltaH\omega=p(p+1)\star\omega.
\end{aligned}
\]
In this case we replace $\omega$ by $\star\omega$. Thus in every case we may assume
\[
\omega\ne0,\qquad \dd\omega=0,\qquad \DeltaH\omega=p(n-p+1)\omega.
\]
Applying \eqref{eq:spectral-bound} to $\omega$ gives $P\omega=0$. Set $\alpha=[p(n-p+1)]^{-1}\dstar\omega$. Then
\begin{equation}\label{eq:primitive}
\begin{aligned}
\dd\alpha&=\frac{\dd\dstar\omega}{p(n-p+1)}=\frac{\DeltaH\omega}{p(n-p+1)}=\omega,\qquad \dstar\alpha=0,\\
\nabla_X\omega&=-\frac1{n-p+1}X^\flat\wedge\dstar\omega=-pX^\flat\wedge\alpha.
\end{aligned}
\end{equation}
Moreover, $\DeltaH\alpha=p(n-p+1)\alpha$. Green's formula and Lemma~\ref{lem:integral} in degree $p-1$ give
\[
\begin{aligned}
0&=\frac{p-1}{p}\int_{\widetilde M}\langle\DeltaH\alpha-p(n-p+1)\alpha,\alpha\rangle\\
&=\frac{p-1}{p}\|\dd\alpha\|^2-(p-1)(n-p+1)\|\alpha\|^2\\
&=\|P\alpha\|^2+\int_{\widetilde M}\left(\langle\mathcal Q(\alpha),\alpha\rangle-(p-1)(n-p+1)|\alpha|^2\right)\\
&\geq0,
\end{aligned}
\]
where the last inequality follows from \eqref{eq:PW} in degree $p-1$ and \eqref{eq:ordering}. Hence
\[
\nabla_X\alpha=\frac1p\iota_X\omega,\qquad \langle\mathcal Q(\alpha),\alpha\rangle=(p-1)(n-p+1)|\alpha|^2.
\]
Together with \eqref{eq:primitive}, this proves the asserted system. It also gives
\begin{equation}\label{eq:closed-pair-norm}
X\bigl(p^2|\alpha|^2+|\omega|^2\bigr)=2p\langle\iota_X\omega,\alpha\rangle-2p\langle X^\flat\wedge\alpha,\omega\rangle=0.
\end{equation}
Thus $p^2|\alpha|^2+|\omega|^2$ is a positive constant. If $\alpha$ vanished on a nonempty open set, then $\omega=\dd\alpha$ would vanish there as well, contradicting this positivity. Hence $\alpha$ is nonzero on a dense subset of $\widetilde M$, and Lemma~\ref{lem:curvature-equality} gives $\R\geq\Id$.
\end{proof}

\begin{lemma}\label{lem:absolute-boundary}
Let $\omega$ be a smooth absolute eigen-$p$-form with eigenvalue $\lambda$. Then
\[
\DeltaH\omega=\lambda\omega\quad\text{in }M,\qquad \iota_\nu\omega=\iota_\nu\dd\omega=0\quad\text{on }\Sigma.
\]
Moreover, $\dd\omega$ and $\dstar\omega$, whenever nonzero, are absolute eigenforms with eigenvalue $\lambda$ in their respective degrees.
\end{lemma}

\begin{proof}
By the variational definition of the absolute Hodge Laplacian, $\iota_\nu\omega=0$, and for every smooth $p$-form $\psi$ with $\iota_\nu\psi=0$, Green's formula gives
\[
\begin{aligned}
0&=\left.\frac{d}{dt}\right|_{t=0}\frac{\|\dd(\omega+t\psi)\|^2+\|\dstar(\omega+t\psi)\|^2}{\|\omega+t\psi\|^2}\\
&=\frac{2}{\|\omega\|^2}\int_M\left(\langle\dd\omega,\dd\psi\rangle+\langle\dstar\omega,\dstar\psi\rangle-\lambda\langle\omega,\psi\rangle\right)\\
&=\frac{2}{\|\omega\|^2}\left[\int_M\langle\DeltaH\omega-\lambda\omega,\psi\rangle+\int_\Sigma\langle\iota_\nu\dd\omega,\psi\rangle\right].
\end{aligned}
\]
Thus $\DeltaH\omega=\lambda\omega$ in $M$ and $\iota_\nu\dd\omega=0$ on $\Sigma$.

At a boundary point, choose an orthonormal frame $e_1,\ldots,e_{n-1}$ with $A(e_a,e_b)=\kappa_a\delta_{ab}$. Since $\nabla_{e_a}\nu=\kappa_a e_a$, differentiating $\iota_\nu\omega=\iota_\nu\dd\omega=0$ along $\Sigma$ gives
\[
\begin{aligned}
\iota_\nu\nabla_{e_a}\omega&=-\iota_{\nabla_{e_a}\nu}\omega=-\kappa_a\iota_{e_a}\omega,\\
\iota_\nu\nabla_{e_a}\dd\omega&=-\iota_{\nabla_{e_a}\nu}\dd\omega=-\kappa_a\iota_{e_a}\dd\omega.
\end{aligned}
\]
Using $\dstar=-\sum_{a=1}^{n-1}\iota_{e_a}\nabla_{e_a}-\iota_\nu\nabla_\nu$, $\iota_\nu\iota_{e_a}=-\iota_{e_a}\iota_\nu$, and $\iota_\nu^2=\iota_{e_a}^2=0$, we obtain
\[
\begin{aligned}
\iota_\nu\dstar\omega&=\sum_{a=1}^{n-1}\iota_{e_a}\iota_\nu\nabla_{e_a}\omega=-\sum_{a=1}^{n-1}\kappa_a\iota_{e_a}^{\,2}\omega=0,\\
\iota_\nu\dstar\dd\omega&=\sum_{a=1}^{n-1}\iota_{e_a}\iota_\nu\nabla_{e_a}\dd\omega=-\sum_{a=1}^{n-1}\kappa_a\iota_{e_a}^{\,2}\dd\omega=0.
\end{aligned}
\]
Consequently,
\[
\iota_\nu\dd\dstar\omega=\lambda\iota_\nu\omega-\iota_\nu\dstar\dd\omega=0.
\]
Since $\DeltaH$ commutes with $\dd$ and $\dstar$, the forms $\dd\omega$ and $\dstar\omega$ satisfy the required eigenvalue equations and boundary conditions.
\end{proof}

\begin{proposition}\label{prop:boundary-pair}
Under the hypotheses of Theorem~\ref{thm:boundary}, suppose that $p\geq2$ and  $\lambda_{1,\mathrm{abs}}^{(p)}(M,g)=p(n-p+1)$. Then $\R\geq\Id$, $A\geq0$, and one of the following holds.
\begin{enumerate}
\renewcommand{\labelenumi}{(\roman{enumi})}
\item There are nonzero forms $\alpha\in\Omega^{p-1}(M)$ and $\omega\in\Omega^p(M)$ with $\dd\alpha=\omega$ on $M$ and $\iota_\nu\alpha=\iota_\nu\omega=0$ on $\Sigma$, satisfying
\[
\nabla_X\alpha=\frac1p\iota_X\omega,\qquad \nabla_X\omega=-pX^\flat\wedge\alpha
\]
for every vector field $X$.
\item We have $n=2p$, and there is a nonzero form $\omega\in\Omega^p(M)$ with $\iota_\nu\omega=\iota_\nu\dd\omega=0$ on $\Sigma$ and
\[
\nabla_X\omega=\frac1{p+1}\iota_X\dd\omega,\qquad \nabla_X\dd\omega=-(p+1)X^\flat\wedge\omega
\]
for every vector field $X$.
\end{enumerate}
\end{proposition}

\begin{proof}
Let $\omega$ be a nonzero absolute eigen-$p$-form with eigenvalue $p(n-p+1)$. If $\dd\omega\ne0$, then \eqref{eq:spectral-bound} gives $n=2p$, and Green's formula gives
\[
\int_M\langle\dstar\dd\omega,\omega\rangle=\|\dd\omega\|^2>0.
\]
In this case, Lemma~\ref{lem:absolute-boundary} shows that $\dstar\dd\omega$ is a nonzero coclosed absolute eigenform with the same eigenvalue, and we replace $\omega$ by it. Applying \eqref{eq:spectral-bound} to the chosen form, we obtain either
\[
\begin{aligned}
&\dd\omega=0,\qquad \nabla_X\omega+\frac1{n-p+1}X^\flat\wedge\dstar\omega=0,\\
\text{or}\quad &n=2p,\qquad \dstar\omega=0,\qquad \nabla_X\omega-\frac1{p+1}\iota_X\dd\omega=0.
\end{aligned}
\]

To apply Lemma~\ref{lem:integral} to $\dstar\omega$ and $\dd\omega$, we first show that their boundary integrals are nonnegative. Let $U=\{x\in\Sigma:\kappa_1(x)<0\}$. Lemma~\ref{lem:absolute-boundary} gives $\iota_\nu\omega=\iota_\nu\dd\omega=0$ on $\Sigma$. Moreover, $\dstar\omega$ is either zero or an absolute eigenform, so $\iota_\nu\dstar\omega=0$ there as well. At a boundary point, choose an orthonormal frame $e_1,\ldots,e_{n-1}$ with $A(e_a,e_b)=\kappa_a\delta_{ab}$. Since $\nabla_{e_a}\nu=\kappa_a e_a$, the conformal Killing equation gives
\[
\begin{aligned}
-\kappa_a\iota_{e_a}\omega=\iota_\nu\nabla_{e_a}\omega=-\frac1{p+1}\iota_{e_a}\iota_\nu\dd\omega+\frac1{n-p+1}e^a\wedge\iota_\nu\dstar\omega=0.
\end{aligned}
\]
If $\omega(e_{i_1},\ldots,e_{i_p})\ne0$ at a point of $U$, then $\kappa_{i_1}=\cdots=\kappa_{i_p}=0$. Hence $\kappa_p\leq0$ and $\kappa_1+\cdots+\kappa_p<0$, contrary to weak $p$-convexity. Thus $\omega=0$ on $U$, and $\nabla_{e_a}\omega=0$ there. The conformal Killing equation and $\iota_\nu\dd\omega=\iota_\nu\dstar\omega=0$ give on $U$
\[
\begin{aligned}
\nabla_\nu\omega&=-\frac1{n-p+1}\nu^\flat\wedge\dstar\omega,\\
\dstar\omega&=-\iota_\nu\nabla_\nu\omega=\frac1{n-p+1}\iota_\nu(\nu^\flat\wedge\dstar\omega)=\frac1{n-p+1}\dstar\omega.
\end{aligned}
\]
Hence $\dstar\omega=0$ and $\nabla_\nu\omega=0$, so $\nabla\omega=\dd\omega=0$ on $U$. The boundary integrals in Lemma~\ref{lem:integral} for $\dstar\omega$ and $\dd\omega$ are therefore nonnegative, since their integrands vanish on $U$ and all principal curvatures are nonnegative on $\Sigma\setminus U$.

Now, suppose first that $\dd\omega=0$, and set $\alpha=[p(n-p+1)]^{-1}\dstar\omega$. The eigenvalue equation and the first alternative give on $M$
\[
\begin{aligned}
\dd\alpha&=\frac{\dd\dstar\omega}{p(n-p+1)}=\frac{\DeltaH\omega}{p(n-p+1)}=\omega,\qquad \dstar\alpha=0,\\
\nabla_X\omega&=-\frac1{n-p+1}X^\flat\wedge\dstar\omega=-pX^\flat\wedge\alpha.
\end{aligned}
\]
In particular, $\alpha\ne0$. By Lemma~\ref{lem:absolute-boundary}, it is an absolute eigenform with the same eigenvalue, and
\[
\iota_\nu\alpha=\frac1{p(n-p+1)}\iota_\nu\dstar\omega=0\quad\text{on }\Sigma.
\]
Its boundary integral in Lemma~\ref{lem:integral} is nonnegative because $\alpha$ is a constant multiple of $\dstar\omega$. Thus Green's formula, Lemma~\ref{lem:integral} in degree $p-1$, and \eqref{eq:PW}--\eqref{eq:ordering} give 
\[
\begin{aligned}
0&=\frac{p-1}{p}\int_M\langle\DeltaH\alpha-p(n-p+1)\alpha,\alpha\rangle\\
&=\frac{p-1}{p}\|\dd\alpha\|^2-(p-1)(n-p+1)\|\alpha\|^2\\
&\geq\|P\alpha\|^2+\int_M\left(\langle\mathcal Q(\alpha),\alpha\rangle-(p-1)(n-p+1)|\alpha|^2\right)\\
&\geq0.
\end{aligned}
\]
Both nonnegative terms vanish, so on $M$
\[
\nabla_X\alpha=\frac1p\iota_X\omega,\qquad \langle\mathcal Q(\alpha),\alpha\rangle=(p-1)(n-p+1)|\alpha|^2,
\]
which proves the system in (i).

On the other hand, suppose now that $n=2p$ and $\dstar\omega=0$. The eigenvalue equation gives
\[
\dstar\dd\omega=p(p+1)\omega,
\]
so $\dd\omega\ne0$. Lemma~\ref{lem:absolute-boundary} gives $\DeltaH(\dd\omega)=p(p+1)\dd\omega$ on $M$ and $\iota_\nu\dd\omega=0$ on $\Sigma$. On each oriented coordinate neighborhood, $\star\dd\omega$ has degree $p-1$. By \eqref{eq:PW} and \eqref{eq:ordering},
\[
\begin{aligned}
\langle\mathcal Q(\dd\omega),\dd\omega\rangle=\langle\mathcal Q(\star\dd\omega),\star\dd\omega\rangle\geq(p-1)\left(\sum_{j=1}^{p+1}\mu_j\right)|\dd\omega|^2\geq(p^2-1)|\dd\omega|^2.
\end{aligned}
\]
Since $\dd\omega$ is closed and its boundary integral is nonnegative, Green's formula and Lemma~\ref{lem:integral} in degree $p+1$ give
\[
\begin{aligned}
0&=\frac{p-1}{p}\int_M\langle\DeltaH(\dd\omega)-p(p+1)\dd\omega,\dd\omega\rangle\\
&=\frac{p-1}{p}\|\dstar\dd\omega\|^2-(p^2-1)\|\dd\omega\|^2\\
&\geq\|P(\dd\omega)\|^2+\int_M\left(\langle\mathcal Q(\dd\omega),\dd\omega\rangle-(p^2-1)|\dd\omega|^2\right)\\
&\geq0.
\end{aligned}
\]
Both nonnegative terms vanish. Since $\dd\omega$ is closed, we obtain on $M$
\[
\nabla_X\dd\omega=-\frac1pX^\flat\wedge\dstar\dd\omega=-(p+1)X^\flat\wedge\omega,\qquad \langle\mathcal Q(\dd\omega),\dd\omega\rangle=(p^2-1)|\dd\omega|^2.
\]
This proves the system in (ii).

The two systems give, respectively,
\[
X\bigl(p^2|\alpha|^2+|\omega|^2\bigr)=0,\qquad X\bigl((p+1)^2|\omega|^2+|\dd\omega|^2\bigr)=0,
\]
as in \eqref{eq:closed-pair-norm}. Each quantity is a positive constant. On $U$, we have $\alpha=\omega=0$ in (i), and $\omega=\dd\omega=0$ in (ii). Hence $U$ is empty and $A\geq0$.

In (i), if $\alpha$ vanished on a nonempty open set, then $\omega=\dd\alpha$ would vanish there, contradicting $p^2|\alpha|^2+|\omega|^2>0$. Thus $\alpha$ is nonzero on a dense subset of $M$, and Lemma~\ref{lem:curvature-equality} gives $\R\geq\Id$.

In (ii), if $\dd\omega$ vanished on a nonempty open set, then $p(p+1)\omega=\dstar\dd\omega$ would vanish there, contradicting $(p+1)^2|\omega|^2+|\dd\omega|^2>0$. Thus $\dd\omega$ is nonzero on a dense subset of $M$. On each oriented coordinate neighborhood, $\star\dd\omega$ has degree $p-1$ and is nonzero on a dense subset, since $|\star\dd\omega|=|\dd\omega|$. Applying Lemma~\ref{lem:curvature-equality} there gives $\R\geq\Id$ on $M$.
\end{proof}

\section{Rigidity on closed manifolds}

\begin{proof}[Proof of Theorem~\ref{thm:closed}]
The eigenvalue estimate was proved in \eqref{eq:spectral-bound}. Suppose first that $p\geq2$ and equality holds. On the universal cover $\widetilde M$, equipped with the lifted metric, fix the nonzero forms $\alpha\in\Omega^{p-1}(\widetilde M)$ and $\omega\in\Omega^p(\widetilde M)$ supplied by Proposition~\ref{prop:closed-pair}. Then $\R\geq\Id$ and
\[
\nabla_X\alpha=\frac1p\iota_X\omega,\qquad \nabla_X\omega=-pX^\flat\wedge\alpha.
\]
Moreover, $\widetilde M$ is compact by Bonnet--Myers. In a local orthonormal frame $e_1,\ldots,e_n$, we have
\[
\dd\alpha=\frac1p\sum_j e^j\wedge\iota_{e_j}\omega=\omega,\qquad \dstar\alpha=-\frac1p\sum_j\iota_{e_j}^{\,2}\omega=0.
\]
Thus $\alpha$ is a special Killing $(p-1)$-form with constant $-p$ in the sense of \cite[Definition~4.1]{Semmelmann}.

By \cite[Theorem~4.8 and pp.~517--519]{Semmelmann}, $\widetilde M$ is the unit sphere, a Sasakian or $3$-Sasakian manifold, a strict nearly K\"ahler manifold of dimension six, or a weak $G_2$-manifold of dimension seven. In the Einstein--Sasakian, $3$-Sasakian, strict nearly K\"ahler and weak $G_2$ cases, the corresponding metric cone is Ricci-flat, and hence $\Ric=(n-1)g$. Thus
\[
\R-\Id\geq0,\qquad \operatorname{tr}(\R-\Id)=\frac12\bigl(\operatorname{Scal}-n(n-1)\bigr)=0,
\]
so $\R=\Id$ and $\widetilde M$ is the unit sphere. If $\widetilde M$ is not round, it is therefore Sasakian and $n$ is odd. In this remaining case, the special Killing forms have odd degree by the same classification. Since $\alpha$ has degree $p-1$, $p$ is even. Finally, since $\widetilde M$ is compact and simply connected with positive curvature operator, it follows from \cite[Theorem~1]{BW} that it is diffeomorphic to $S^n$.

Suppose now that $p=1$ and equality holds. Let $\omega$ be a nonzero eigenform on $M$ with $\DeltaH\omega=n\omega$. Since $n\geq3$, \eqref{eq:spectral-bound} gives $\dd\omega=0$ and $P\omega=0$. Let $f=n^{-1}\dstar\omega$. Then
\[
\begin{aligned}
\dd f&=\frac1n\dd\dstar\omega=\frac1n\DeltaH\omega=\omega,\\
\nabla^2f&=\nabla\omega=-\frac1n(\dstar\omega)g=-f g.
\end{aligned}
\]
Since $\omega\ne0$, the function $f$ is nonconstant. Obata's theorem \cite{Obata} gives that $M$ is the unit sphere, and hence $\R=\Id$.
\end{proof}

\section{Rigidity on manifolds with boundary}

For the following lemma, suppose that $p\geq2$ and equality holds in Theorem~\ref{thm:boundary}. Fix the forms supplied by Proposition~\ref{prop:boundary-pair}. Then $\R\geq\Id$ and $A\geq0$. By \cite[Theorem~2(ii)]{Wu}, $M$ is diffeomorphic to a closed ball. In particular, $M$ is simply connected, and $\Sigma$ is connected and simply connected since $n\geq4$.

\begin{lemma}\label{lem:boundary-obata}
If $\Sigma$ is not isometric to the unit sphere, then the first Dirichlet eigenvalue $\lambda_1^D(M)$ of the Laplacian on functions satisfies $\lambda_1^D(M)\leq n$.
\end{lemma}

\begin{proof}
On the metric cone $(\widehat M,\widehat g)=((0,\infty)\times M,\dd r^2+r^2g)$, define
\[
\widehat\Phi=\begin{cases}
r^{p-1}\dd r\wedge\alpha+\dfrac{r^p}{p}\omega,&\text{if Proposition~\ref{prop:boundary-pair}(i) holds},\\[5pt]
r^p\dd r\wedge\omega+\dfrac{r^{p+1}}{p+1}\dd\omega,&\text{if Proposition~\ref{prop:boundary-pair}(ii) holds},
\end{cases}
\]
where the forms are pulled back from $M$. In alternative (i), $\dd\alpha=\omega$ and $\alpha$ is a special Killing $(p-1)$-form with constant $-p$. In alternative (ii), $\omega$ is a special Killing $p$-form with constant $-(p+1)$. Thus \cite[Lemma~4.3]{Semmelmann} gives $\widehat\nabla\widehat\Phi=0$ in both cases, up to the boundary.

For each $x\in\widehat M$, define
\[
K_x=\{v\in T_x\widehat M:\iota_v\widehat\Phi_x=0\}.
\]
We claim that $K$ is a parallel subbundle of $T\widehat M$. Let $\gamma:[0,1]\to\widehat M$ be a smooth curve, and let $Z$ be parallel along $\gamma$, with $Z(0)\in K_{\gamma(0)}$. Since $\widehat\nabla\widehat\Phi=0$, we have
\[
\begin{aligned}
\frac{\dd}{\dd t}|\iota_Z\widehat\Phi|^2&=2\left\langle\iota_{\widehat\nabla_{\dot\gamma}Z}\widehat\Phi+\iota_Z\widehat\nabla_{\dot\gamma}\widehat\Phi,\iota_Z\widehat\Phi\right\rangle=0,
\end{aligned}
\]
giving $|\iota_{Z(t)}\widehat\Phi_{\gamma(t)}|^2=|\iota_{Z(0)}\widehat\Phi_{\gamma(0)}|^2=0$. Thus $Z(t)\in K_{\gamma(t)}$. Applying this also to the reverse curve shows that parallel transport maps $K_{\gamma(0)}$ onto $K_{\gamma(1)}$. Hence the spaces $K_x$ have constant dimension and form a smooth subbundle $K$ of $T\widehat M$. For a smooth section $Z$ of $K$ and a vector field $X$, we have
\[
\iota_{\widehat\nabla_XZ}\widehat\Phi=\widehat\nabla_X(\iota_Z\widehat\Phi)-\iota_Z\widehat\nabla_X\widehat\Phi=0.
\]
Thus $\widehat\nabla_XZ$ takes values in $K$, which proves the claim.

Let $\widehat\Sigma=(0,\infty)\times\Sigma$ and let $\widehat\nu=r^{-1}\nu$ be its outward unit normal. Since $\dd r(\widehat\nu)=0$, the boundary conditions in Proposition~\ref{prop:boundary-pair} give
\[
\iota_{\widehat\nu}\widehat\Phi=\begin{cases}
-r^{p-2}\dd r\wedge\iota_\nu\alpha+\dfrac{r^{p-1}}p\iota_\nu\omega=0,&\text{if (i) holds},\\[5pt]
-r^{p-1}\dd r\wedge\iota_\nu\omega+\dfrac{r^p}{p+1}\iota_\nu\dd\omega=0,&\text{if (ii) holds}.
\end{cases}
\]
Thus $\widehat\nu(x)\in K_x$ for $x\in\widehat\Sigma$. For $v\in K_x$, we have
\[
\begin{aligned}
\left\langle v-\langle v,\widehat\nu\rangle\widehat\nu,\widehat\nu\right\rangle=0,\qquad
\iota_{v-\langle v,\widehat\nu\rangle\widehat\nu}\widehat\Phi=\iota_v\widehat\Phi-\langle v,\widehat\nu\rangle\iota_{\widehat\nu}\widehat\Phi=0.
\end{aligned}
\]
Hence $v-\langle v,\widehat\nu\rangle\widehat\nu\in K_x\cap T_x\widehat\Sigma$, and
\[
K_x=\mathbb R\widehat\nu(x)\oplus(K_x\cap T_x\widehat\Sigma)\quad(x\in\widehat\Sigma).
\]
In particular, $K\cap T\widehat\Sigma$ is a subbundle with $\operatorname{rank}(K\cap T\widehat\Sigma)=\operatorname{rank}K-1$. The induced Levi-Civita connection on $\widehat\Sigma$ is the tangential projection of $\widehat\nabla$. For vector fields $X,Z$ tangent to $\widehat\Sigma$, with $Z$ taking values in $K$, this gives
\[
\begin{aligned}
\iota_{\nabla_X^{\widehat\Sigma}Z}\widehat\Phi=\iota_{\widehat\nabla_XZ}\widehat\Phi-\langle\widehat\nabla_XZ,\widehat\nu\rangle\iota_{\widehat\nu}\widehat\Phi=\widehat\nabla_X(\iota_Z\widehat\Phi)-\iota_Z\widehat\nabla_X\widehat\Phi=0.
\end{aligned}
\]
Consequently, $\nabla_X^{\widehat\Sigma}Z$ also takes values in $K\cap T\widehat\Sigma$, so this subbundle is parallel for the induced connection.

Since $\Sigma$ is compact, simply connected and not the unit sphere, $\widehat\Sigma$ has irreducible holonomy by \cite[Lemma~2]{Bar}. Therefore its parallel tangent subbundle $K\cap T\widehat\Sigma$ is either zero or all of $T\widehat\Sigma$. The latter would give $K_x=T_x\widehat M$ on the boundary and hence $\widehat\Phi_x=0$, contrary to the constant positive length of $\widehat\Phi$. Thus
\[
K_x\cap T_x\widehat\Sigma=\{0\},\qquad K_x=\mathbb R\widehat\nu(x)\quad(x\in\widehat\Sigma).
\]
Since the rank of $K$ is constant, $K$ is a parallel line bundle on $\widehat M$.

Let $\pi_K$ be the orthogonal projection onto $K$, and define the vector field $V=\pi_K\partial_r$ on $\{1\}\times M$. Since the metric connection $\widehat\nabla$ preserves $K$ and $K^\perp$, we have $\widehat\nabla\pi_K=0$. For vector fields $X,Y$ tangent to $M$, the cone connection at $r=1$ gives
\[
\begin{aligned}
\widehat\nabla_XV&=\pi_K\widehat\nabla_X\partial_r=\pi_KX,\\
\widehat\nabla_X\widehat\nabla_YV-\widehat\nabla_{\nabla_XY}V&=\pi_K(\widehat\nabla_XY-\nabla_XY)=-g(X,Y)V.
\end{aligned}
\]
The vector field $V$ is not identically zero, since otherwise the first identity and $V=\pi_K\partial_r$ would give $\pi_K=0$. On $\Sigma$, we have $K=\mathbb R\nu$, so $V=\langle\partial_r,\nu\rangle\nu=0$.

Lastly, for $\varepsilon>0$, the function $f_\varepsilon=(|V|^2+\varepsilon^2)^{1/2}-\varepsilon$ is smooth up to the boundary and vanishes on $\Sigma$. In a local orthonormal frame $e_1,\ldots,e_n$ on $M$, Cauchy--Schwarz and integration by parts give
\[
\begin{aligned}
\lambda_1^D(M)\int_M f_\varepsilon^2&\leq\int_M|\nabla f_\varepsilon|^2\\
&=\int_M\sum_j\frac{\langle\widehat\nabla_{e_j}V,V\rangle^2}{|V|^2+\varepsilon^2}\\
&\leq\int_M\sum_j|\widehat\nabla_{e_j}V|^2\\
&=-\int_M\sum_j\left\langle\widehat\nabla_{e_j}\widehat\nabla_{e_j}V-\widehat\nabla_{\nabla_{e_j}e_j}V,V\right\rangle\\
&=n\int_M|V|^2.
\end{aligned}
\]
Since $0\leq|V|-f_\varepsilon\leq\varepsilon$ and $V\not\equiv0$, letting $\varepsilon\downarrow0$ proves the assertion.
\end{proof}

\begin{proof}[Proof of Theorem~\ref{thm:boundary}]
The eigenvalue estimate was proved in \eqref{eq:spectral-bound}. 

Suppose first that $p=1$ and equality holds. Let $\omega$ be a nonzero absolute eigen-$1$-form with $\DeltaH\omega=n\omega$. Since $n\geq3$, \eqref{eq:spectral-bound} gives $\dd\omega=0$ and $P\omega=0$. Set $f=n^{-1}\dstar\omega$. Then
\[
\begin{aligned}
\dd f&=\frac1n\dd\dstar\omega=\frac1n\DeltaH\omega=\omega,\\
\nabla^2f&=\nabla\omega=-\frac1n(\dstar\omega)g=-f g,\\
\partial_\nu f&=\iota_\nu\omega=0\quad\text{on }\Sigma.
\end{aligned}
\]
Since $\omega\ne0$, the function $f$ is nonconstant. The Neumann rigidity theorem in \cite{Escobar} implies that $M$ is the unit hemisphere.

Suppose now that $p\geq2$, equality holds, and $\Sigma$ is not the unit sphere. Proposition~\ref{prop:boundary-pair} gives $\Ric\geq(n-1)g$ and $A\geq0$. Reilly's estimate \cite{Reilly}, recalled in \cite[Theorem~5]{HW}, and Lemma~\ref{lem:boundary-obata} give
\[
n\leq\lambda_1^D(M)\leq n.
\]
The equality case of Reilly's theorem implies that $M$ is the unit hemisphere, a contradiction. Thus $\Sigma$ is the unit sphere, and \cite[Theorem~2]{HW} implies that $M$ is the unit hemisphere.

Conversely, let $j:S^n_+\hookrightarrow\mathbb R^{n+1}$ be the inclusion of the unit hemisphere $\{x_{n+1}\geq0\}$, and let
\[
\omega=j^*(\dd x_1\wedge\cdots\wedge\dd x_p),\qquad \alpha=\frac1p j^*\bigl(\iota_x(\dd x_1\wedge\cdots\wedge\dd x_p)\bigr),
\]
where $x$ is the position vector. The form $\omega$ is nonzero at the north pole and satisfies $\dd\omega=0$. Since $\nu=-e_{n+1}$ on the equator, $\iota_\nu\omega=\iota_\nu\dd\omega=0$. The Gauss formula gives
\[
\nabla_X\alpha=\frac1p\iota_X\omega,\qquad \nabla_X\omega=-pX^\flat\wedge\alpha.
\]
Consequently, in a local orthonormal frame $e_1,\ldots,e_n$,
\[
\begin{aligned}
\dd\alpha&=\frac1p\sum_j e^j\wedge\iota_{e_j}\omega=\omega,\\
\dstar\omega&=p\sum_j\iota_{e_j}(e^j\wedge\alpha)=p\left(n\alpha-\sum_j e^j\wedge\iota_{e_j}\alpha\right)=p(n-p+1)\alpha,\\
\DeltaH\omega&=\dd\dstar\omega=p(n-p+1)\dd\alpha=p(n-p+1)\omega.
\end{aligned}
\]
Thus the hemisphere attains the lower bound.
\end{proof}


\end{document}